\documentclass[letterpaper,twoside]{amsart}

\usepackage[T1]{fontenc}
\usepackage[utf8]{inputenc}

\usepackage{amsmath}
\usepackage{amsfonts}
\usepackage{amssymb}
\usepackage{enumerate}
\usepackage{amsthm}
\usepackage{dsfont}
\usepackage{tikz}

\newcommand{\titsb}[1]{\ensuremath{\partial_T #1}}
\newcommand{\titsm}[2]{\ensuremath{d_T(#1,#2)}}
\newcommand{\tangle}[2]{\ensuremath{\angle(#1,#2)}}
\newcommand{\bndry}[1]{\ensuremath{#1\left(\infty\right)}}
\newcommand{\coneb}[1]{\ensuremath{\partial_\infty #1}}
\newcommand{\cat}[1]{\ensuremath{\textnormal{CAT}(#1)}}
\newcommand{\catz}{\ensuremath{\textnormal{CAT}(0)}}

\newcommand{\R}{\ensuremath{\mathbf{R}}}
\newcommand{\N}{\ensuremath{\mathbf{N}}}

\newcommand{\isom}[1]{\ensuremath{\textnormal{Isom}\left(#1\right)}}

\newcommand{\ball}[2]{\ensuremath{\textnormal{B}\left(#1,#2\right)}}
\newcommand{\buse}[3]{\ensuremath{B_{#1,#2}\left(#3\right)}}
\newcommand{\disthaus}[2]{\ensuremath{d_{\mathcal{H}}\left(#1,#2\right)}}
\newcommand{\contoverconst}{\ensuremath{\mathrm{C}\left(X,\R\right)/\R}}
\newcommand{\can}[1]{\ensuremath{#1^{\prime}}}
\newcommand{\para}[1]{\ensuremath{\mathcal{P}\left(#1\right)}}
\newcommand{\ultra}[1]{T^{#1}}
\newcommand{\distinf}[2]{\ensuremath{\mathrm{d}_{\infty}(#1,#2)}}
\renewcommand{\SS}{{SS}}

\theoremstyle{definition}
\newtheorem{defi}{Definition}
\newtheorem{prop}{Proposition}
\newtheorem{lemma}{Lemma}
\newtheorem{thm}{Theorem}
\newtheorem{cor}{Corollary}
\newtheorem{rmk}{Remark}
\newtheorem{ex}{Example}
\newtheorem{quest}{Question}

\newcounter{deux}
\newcounter{un}

\begin{document}
\title{Tits compact CAT(0) spaces}
\author[A.~Bosch\'e]{Aur\'elien Bosch\'e}
\address{Institut Fourier\\ 
100 rue des maths\\
38402 St Martin d'Hères\\
FRANCE}
\email{aurelb@ujf-grenoble.fr }
\begin{abstract}
A standard result in \catz\ geometry states that a cocompact \catz\ space has a discrete Tits Topology (or equivalently the Tits metric is the only metric on the boundary that takes only the two values $0$ and $+\infty$) if and only if it is hyperbolic in the sense of Gromov. So in the category of cocompact \catz\ spaces, having the coarsest possible Tits topology (indeed a Hausdorff topology is always finer than the discrete topology) is equivalent to being Gromov hyperbolic. In this paper we study the opposite situation, i.e. the case where the Tits topology is the finest possible, that is to say when it coincides with the cone topology of Eberlein (remember the Tits topology is always coarser that the cone topology), or equivalently when it is compact. We show some results that apply in this general setting and that imply under some mild assumptions (like the existence of a hyperbolic isometry or the existence of a cocompact parabolic group) that it splits with a (non trivial) euclidean factor. Using recent results of P.E. Caprace and N. Monod, we then obtain that cocompact geodesically complete \catz\ spaces with compact Tits boundary are flat. We strongly suspect that if the space is not geodesically complete, then it is still ``as flat as it can be'', meaning that its canonical boundary minimal subspace is flat (or equivalently its Tits boundary is a metric unit sphere). We show that this is the case if the space admits a cocompact action by semi-simple isometries with locally finite stabilisers. In the final section we investigate the existence of hyperbolic isometries for cocompact \catz\ spaces, and show there existence when the space contains no convex flat planes.
\end{abstract}
\maketitle
\setcounter{section}{-1}
\section{Initial remark}
 The author observed just after completing this article that the so called $\pi$-convergence property (i.e. lemma $18$ of \cite{papasoglu2009boundaries}) is valid for non proper actions (the proof does not make use of this assumption). This would give other proofs for the lemma \ref{lem:split} and the proposition \ref{prop:fixflat}. More importantly it is possible that for the same reason the results (or some results) of \cite{guralnik2011atransversal} hold for cocompact non proper actions too. If so, this would supersed the propositions \ref{prop:geodcompflat} and \ref{prop:semisimlocalfinflat} below.

\section{Introduction}
A metric space $(X,d)$ is called \catz\ if it is geodesic and all its geodesic triangles are not fatter than there euclidean comparison triangle. For general reference on \catz\ spaces, see \cite{BridsonHaefliger201011} and \cite{Ballmann200402}. A geodesic is an isometry $\varphi:I\rightarrow X$ where $I$ is an interval of \R. If $I=\R^+$ (resp. $I=\R$) then $\varphi$ is called a geodesic ray (resp. geodesic line). We emphasis that a \catz\ space might admit no ray (take any bounded convex subset of $\R^n$, any ball of any arbitrary \catz\ space, ...). The Hausdorff distance between two subsets $A$ and $B$ of a metric space $X$ is defined by (for $Y\subset X$ and $r>0$, \ball{Y}{r} designates $\cup_{y\in Y} B(y,r)$  where the balls are those of $X$):
\begin{equation*}
  \disthaus{A}{B}=\inf_{r>0}\left\{B\subset\ball{A}{r}\textrm{ and }A\subset\ball{B}{r}\right\}.
\end{equation*}
We call two rays $\varphi$ and $\psi$ asymptotic if there images are at bounded Hausdorff distance. This is equivalent to saying that the function $t\mapsto d(\varphi(t),\psi(t))$ is uniformly bounded. Being asymptotic is an equivalent relation. The quotient set of all geodesic rays by this equivalent relation is written \bndry{X}, and the class of a ray $\varphi$ is $\varphi(\infty)$. If $X$ is complete (i.e. if every Cauchy sequence converges), one can find a representative of any class of asymptotic rays radiating from any given point $x\in X$, and this representative is necessarily unique. Since we will only be interested in proper spaces (i.e. spaces whose closed balls are compact), we remark that such spaces are complete. Suppose $x\in X$ is now fixed, and give the set of rays emanating from $x$ the topology of uniform convergence on compact sets. This gives a topology on \bndry{X}\ and one can show that this topology is independent of $x$. We write \coneb{X}\ for this topological space and call it the boundary at infinity (endowed with the cone topology if we want to be specific).

There is another description of the boundary in terms of horofunctions. Let us define $\Theta: X\rightarrow \contoverconst; x\mapsto [y\mapsto d(y,x)]$ where \contoverconst\ is the quotient of continuous functions on $X$ quotiented by the constants with the quotient topology induced from the compact open topology on the set of continuous functions. Then \bndry{X} can be canonically identified with the boundary of $\Theta(X)$ in \contoverconst. Indeed if $\varphi$ is a ray in $X$ then the following function is well defined and $1$-lipschitz on $X$:
\begin{equation*}
  \label{eq:busedef}
  x\mapsto\lim_{t\rightarrow+\infty}d(x,\varphi(t))-t.
\end{equation*}
 We write \buse{[\varphi]}{\varphi(0)}{x} for this limit and call the associated function of $x$ the Busemann function based at $[\varphi]$ with reference point $\varphi(0)$. The class in \contoverconst\ of this limit only depends on the asymptotic class $[\varphi]$ of the ray $\varphi$. When $X$ has non positive curvature, all the elements of the boundary of $\Theta(X)$ in \contoverconst\ can thus be obtained. In general we obtain two different compactifications (see \cite{paeng2007tits} for example for the case of manifolds without conjugate points).
 
 The horosphere based at $p\in\bndry{X}$ through $x\in X$ is the level set through $x$ of any Busemann function based at $p$, i.e. it is $b^{-1}(b(x))$ for any such Busemann function. A (closed) horoball based at $p$ is any set of the form $b^{-1}(]-\infty,a])$ where $a\in\R$ and $b$ is a Busemann function based at $p$.

\coneb{X} is compact and second countable (it is readily seen that if $X$ is complete not a priori proper, if \coneb{X} is compact and second countable and if all geodesics in $X$ can be extended to geodesic rays, then  $X$ is itself must be proper). From now on, all the \catz\ spaces we will encounter will be proper (this is equivalent to being locally compact for geodesic spaces). Choose $x\in X$ and let $\varphi$ and $\psi$ be two rays emanating from $x$. Then $2\cdot\arcsin(d(\varphi(t),\psi(t))/2t)$ converges to a real as $t$ goes to $+\infty$. Moreover this real depends only on the asymptotic classes of $\varphi$ and $\psi$, and not on $x$. We write $\tangle{\varphi(\infty)}{\psi(\infty)}$ for this limit. This defines a metric on $\bndry{X}$ and we shall refer to it as the Tits angle. The topology induced by  this metric is always finer that the cone topology, that is to say the identity $i:\bndry{X}\rightarrow\bndry{X}$ induces a continuous bijection $\bar{i}:(X,\angle)\rightarrow\coneb{X}$. We define the Tits metric, and write $\titsm{\cdot}{\cdot}$, to be the inner metric associated to the Tits angle. Hence this metric is defined by:
\begin{displaymath}
  \forall \varphi,\psi\in\bndry{X},\ \titsm{\varphi}{\psi}=\inf_{\gamma}\ell(\gamma),
\end{displaymath}
where $\gamma:[0,1]\rightarrow\bndry{X}$ ranges  over all the continuous path in $\bndry{X}$ (for the topology induced by the Tits angle) that connect $\varphi$ and $\psi$, and $\ell(\gamma)$ is the length of $\gamma$ with respect to the Tits angle defined by:
\begin{displaymath}
  \ell(\gamma)=\sup_{n,\ (t_i)_{1\leq 0\leq n}} \sum_{i=0}^{n-1}\angle(\gamma(t_i),\gamma(t_{i+1})),
\end{displaymath}
where the supremum runs over all $n>0$ and all $0=t_0<t_1<\cdot<t_n=1$. One can show that the balls of radius $r<\pi$ are the same for the Tits angle and the Tits metric (Part \Roman{deux} proposition $9.21$  of \cite{BridsonHaefliger201011}), so that they generate the same topology. We will refer to this topology as the Tits topology, and abuse the vocabulary and call it the Tits topology of $X$.

 Let us briefly remind the classification of isometries in \catz\ spaces. An isometry of a \catz\ space $X$ is called elliptic if it has a fixed point, hyperbolic if it translates some geodesic, and parabolic if it is neither elliptic nor hyperbolic. Isometries naturally act on the boundary \bndry{X} and it can be shown that parabolic isometries fix at least one point in it (we make use of properness here). Non-parabolic isometries are called semi-simple. A geodesic translated by a hyperbolic isometry is called an axis of this isometry. Moreover if $Y\subset X$ is a convex subset stabilised by an isometry $\phi$, then $\phi$ and its restriction to $Y$ have same type. Now let $\phi$ be an isometry fixing a point $p$ at infinity: if $b$ is a Busemann function based at $p$ and $\phi^{*}b$ is the action of $\phi$ on $b$ then $b-\phi^{*}b$ is a constant that only depends of the class of $b$ (i.e. on $p$). This constant is called the Busemann character of $\phi$ at $p$.

We say that $X$ is Tits compact if the Tits topology is compact. For example, it is well known that $\R^n$ with the usual flat metric has a Tits boundary isometric to the standard metric sphere of radius one, and so is Tits compact. More trivially, bounded convex subsets of any \catz\ space (remark that there always exist such subspaces since the balls of \catz\ spaces are convex) have empty boundary, and hence are Tits compact. The following example shows that there exist non flat \catz\ spaces with the geodesic extension property (i.e. such that every geodesic is part of a geodesic line) that are Tits compact: 
\begin{ex}\label{ex:conesphere}
  Take $r\ge 1$ and $n\in\N^*$. Let $\mathcal{S}_r$ be the metric sphere of radius $r$ and dimension $n-1$ and construct the euclidean cone $X$ over $\mathcal{S}_r$ (see section $5$ Part \Roman{un} of \cite{BridsonHaefliger201011} for the definition and general properties of cones over metric spaces). Since the curvature of $\mathcal{S}_r$ is bounded above by $1$ and $\mathcal{S}_r$ is simply connected, $\mathcal{S}_r$ is metrically a $\cat{1}$ space. Then, according to Berestovskii' s theorem (see Part \Roman{deux} proposition $3.14$ in \cite{BridsonHaefliger201011}), $X$ is a \catz\ space. It is easy to see that it is locally compact, has the geodesic extension property, and that its Tits boundary is isometric to $\mathcal{S}_r$. To prove the last affirmation, take $y,z$ on the boundary $\mathcal{S}_r$ of $X$. Then the rays asymptotic to $y$ and $z$ are respectively $t\mapsto (y,t)$ and $t\mapsto (z,t)$, and
  \begin{align*}
    \forall t>0,\ d_X((y,t),(z,t))/2t&=\frac{1}{\sqrt{2}}\sqrt{1-\cos\min\left(d_{\mathcal{S}_r}(y,z),\pi\right)}\\
    &=\frac{1}{\sqrt{2}}\sqrt{2\sin^2\bigg(\min\left(d_{\mathcal{S}_r}(y,z),\pi\right)/2\bigg)}\\
    &=\sin\bigg(\min\left(d_{\mathcal{S}_r}(y,z),\pi\right)/2\bigg),
  \end{align*}
so that $\titsm{y}{z}=d_{\mathcal{S}_r}(y,z)$. In particular $X$ is Tits compact. Let us now turn our attention to the isometries of $X$. Since the cone construction is functorial, every isometry of $\mathcal{S}_r$ extends to an isometry of $X$. This isometry is always elliptic since it fixes the apex of the cone (which we shall call $0$ henceforth). Reciprocally, it is easy to see that the isometries fixing $0$ are uniquely induced by isometries of the Tits boundary $\mathcal{S}_r$. If $r=1$, $X$ is the Hilbert space of dimension $n$, so that there are plenty of other isometries, including translations (Clifford translations in the terminology of \catz\ spaces) in any direction (remark that the actions at infinity of all those translations are trivial, so one cannot hope to recover any of them from its continuation to \titsb{X}). Let us now turn to the case where $r$ is bigger than $1$. Then the inspection of the geodesics in $X$ shows that $0$ is the only point where geodesics branch. So isometries of $X$ must fix $0$. The correspondence between isometries of the Tits boundary of $X$ (naturally isometric to the initial $\mathcal{S}_{r}^{n-1}$) and isometries of $X$ is then easily proved to be a isomorphism of topological groups. So the group of isometries of $X$ is a whole lie group of positive dimension. However, it is compact, so can be considered relatively small and with rather trivial dynamics. 
\end{ex}
\begin{defi}
  A subspace $Y$ of a \catz\ space $X$ is geodesic if for all points $x,y\in Y$, the geodesic in $X$ joining those two points lies inside $Y$.
\end{defi}
\begin{defi}
  A (half) flat plane in $X$ is a convex subset isometric to a euclidean (half) plane.   A (half) flat strip in $X$ is a convex subset isometric to a euclidean (half) strip, i.e. to $[0,a]\times\R$ ($[0,a]\times\R^+$) for some $a>0$.
\end{defi}
The next lemma is  classical and can be found in every textbook on non positively curved spaces like \cite{BridsonHaefliger201011} or \cite{Ballmann200402} and is nicknamed the flat strip theorem:
\begin{lemma}
  Let $X$ be a \catz\ space a let $\ell$ be a geodesic line in $X$ (i.e. $\ell$ is a subset of $X$ isometric to the real line). We call another line $l$ parallel to $\ell$ if they lie at bounded Hausdorff distance from one another. Then if $\ell$ and $l$ are parallel they bound a totally geodesic flat strip in $X$ (i.e. there exists an isometry $\varphi:\R\times[0,a]\rightarrow X$ such that $\varphi(\R\times{0})=\ell$ and $\varphi(\R\times{a})=l$ where $a>0$). Moreover the set of parallel lines to $\ell$ form a totally geodesic subspace that decompose has a product $\R\times Y$ where $Y$ is a Hadamard space and the first factor corresponds to the parallel lines in $X$ to $\ell$ (i.e. the parallel lines to $\ell$ are exactly the $\R\times\{a\}$ where $a\in Y$).
\end{lemma}
 We say that a group $G$ acts cocompactly on a space $X$ if there exists a compact subset $K$ of $X$ whose translates by $G$ cover $X$. In the same context we say that $G$ acts properly if for every compact subset $C$ of $X$ the number of elements $g$ in $G$ satisfying $g\cdot C\cap C\neq\emptyset$ is finite (we shall in time modify this definition). Since the spaces we will be concerned with will be proper, this  is equivalent to the various usual other definitions of properness (for discrete groups). Finally an action is geometric if it is both cocompact and proper.

\section{Minimal subspaces}

We remind results that can be found in \cite{caprace2009isometry}. The reader should refer to this paper for details  and proofs.
\begin{defi}
  Let $X$ be a \catz\ space and $Y\subset X$ a convex subset. We say that $Y$ has full boundary (in $X$) if $\bndry{Y}=\bndry{X}$. Moreover, we say that $Y$ is boundary minimal if it is closed, convex, has full boundary, and is minimal for this property (the partial order considered is the inclusion).
\end{defi}
The existence of boundary minimal subset is not clear. However, we have:
\begin{prop}
  If a \catz\ space $X$ is cocompact, then it admits boundary minimal subsets and all those subspaces are isometric (they are even parallel but we will not need it). Moreover, one can canonically assign to $X$ a subspace $\can{X}$ that is boundary minimal and fixed by all isometries of $X$. This space is called the canonical boundary minimal subspace of $X$.
\end{prop}
\begin{rmk}
  \can{X} can be defined in greater generality but we shall not use it.
\end{rmk}
The product of two \catz\ spaces $X_1$ and $X_2$ is again a \catz\ space. The Tits boundary of the product depends only on the Tits boundary of each factor and can be constructed explicitly. The construction is called joining (we slightly modify the usual definition since with the usual definition of \cite{BridsonHaefliger201011} one should take the inner metric associated to the joining in order to get the Tits metric of the product). The join of two spaces $Y_1$ and $Y_2$ is written $Y_1\star Y_2$. Hence, almost by definition we have $\titsb{\left(X_1\times X_2\right)}=\titsb{X_2}\star\titsb{X_1}$. Reciprocally if the boundary of $X$ is a join and if $X$ is geodesically complete then $X$ is a product. Without assuming geodesic completness, we have:
\begin{prop}\label{prop:bndrysphere}
  Let $X$ be a cocompact \catz\ space. Then \can{X} is flat if and only if \titsb{X} is a metric unit sphere.
\end{prop}
\begin{proof}
Since \can{X} has full boundary in $X$, the sufficient part is trivial. We now suppose that \titsb{X} is a metric unit sphere of dimension $n$.

  We can assume that $X$ is itself boundary minimal. Decompose $X$ as a product $\R^k\times Y$ where $Y$ is a \catz\ space that admits no euclidean factor. This decomposition exists, is unique, and every isometry of $X$ respects the splitting so that $Y$ is also cocompact. It must also be boundary minimal since so is $X$. Suppose that $X$ is not flat. Then $Y$ is not compact by minimality of $X$.  Moreover $\titsb{Y}$ is a unit sphere (it is the set of points at constant distance $\pi/2$ from the subsphere \titsb{\R^k} of \titsb{X}).

 Since it is also cocompact it admits a geodesic line $\varphi$ (a non compact proper geodesic space always admits a ray. It is then easy to obtain a line by cocompactness and properness of $X$). Let $x$ be a point on $\varphi$. Now $\varphi(+\infty)$ must be the opposite point of $\varphi(-\infty)$ in the metric sphere \titsb{Y} since it is the only point that lies at distance not lower than $\pi$ from $\varphi(-\infty)$. We shall consider the two endpoints of $\varphi$ as the poles of the sphere. If the boundary has only two points, then $Y=\varphi(\R)$ by minimality and $Y$ is flat, contradicting the hypothesis. Pick a point $y\in\bndry{Y}$ on the equator, and let $\psi$ be the ray from $x$ to $y$. We hence have $\angle(\varphi(\pm\infty),y)=\pi/2$ and:
  \begin{align*}
    \pi=\angle_{x}\left(\varphi(+\infty),\varphi(-\infty)\right)&\leq \angle_{x}\left(\varphi(+\infty),y\right)+\angle_{x}\left(y,\varphi(-\infty)\right)\\
    &\leq \angle\left(\varphi(+\infty),y\right)+\angle\left(y,\varphi(-\infty)\right)=\pi,
  \end{align*}
So that we have equality everywhere. In particular we have
\begin{displaymath}
\angle_{x}\left(\varphi(\pm\infty),y\right)=\angle\left(\varphi(\pm\infty),y\right)=\pi/2.
\end{displaymath}
 From this we deduce that the rays $\varphi(\R^+)$ and $\psi$ (resp. $\varphi(\R^-)$ and $\psi$) span a convex flat quadrant. In fact it is easy to see that those two quadrants form a convex half flat plane with edge $\varphi$ admitting $y$ as boundary point (maybe the easiest way to see this is to remark that we chose $x$ arbitrarily). Hence the set \para{\varphi} of parallel lines to $\varphi$ (which is a closed convex subset of $X$) has full boundary so that $Y=\para{\varphi}$ by minimality. But we know that \para{\varphi} splits as $\R\times Z$ where $Z$ is a \catz\ space. This contradicts the definition of $Y$ and $X$ must be flat.
\end{proof}
\begin{quest}
  Are there non flat boundary minimal \catz\ spaces with Tits boundary isometric to a metric sphere (without assuming geodesic completness of course)?
\end{quest}
\section{Geometric group actions and Tits compactness}\label{sec:geomact}
We only include this section for the curiosity of the reader since we will prove a stronger result in section \ref{sec:geomactbetter}.

 The reference for this section is \cite{guralnik2011atransversal} and the reader should refer to it for proofs and details. We will also use the notations and definitions of this article. The definition $4.2$ in this paper does not always make sense since the Tits metric could be infinite. However we will only make use of the Theorem $3.22$ (the folding lemma) so this will not be a problem for us. Recall theorem $3.22$ of this article:
\begin{thm}[Folding lemma]\label{thm:gural}
   Suppose $G$ acts geometrically on a \catz\ space $X$ and let $d$ denote the geometric dimension of $\titsb{X}$. Then for every $(d+1)$-flat $F_0\subseteq X$ there exist $\omega\in\beta G$ and a $(d+1)$-flat $F\subseteq X$ such that $\ultra{\omega}$ maps all of $\bndry{X}$ onto $\SS=\bndry{F}$, with $\bndry{F_0}$ mapped isometrically onto $\SS$.
\end{thm}
\begin{rmk}
  The flat $F_0$ exists, see the theorem C in \cite{kleiner1999local} and in particular the equivalent statements $3$ and $5$.
\end{rmk}
But we have:
\begin{lemma}
  Let $X$ be a Tits compact \catz\ space on which a discrete group $G$ acts. Then for every $\omega\in\beta G$, $x\mapsto\ultra{\omega}(x)$ is distance preserving on \bndry{X}.
\end{lemma}
\begin{proof}
  Let us pick $x,y\in\titsb{X}$. Then for all $n>0$ the set of all $g\in G$ such that $\titsm{g(x)}{\ultra{\omega}(x)}<1/n$ and $\titsm{g(y)}{\ultra{\omega}(y)}<1/n$ is not empty (it has $\omega$ probability $1$, because the Tits balls of radius $1/n$ around $g(x)$ and $g(y)$ are open in the cone topology), so we can fix such an element $h_n\in G$ for each $n$. We have for all $n\in\N$:
  \begin{equation*}
   |\titsm{\ultra{\omega}(x)}{\ultra{\omega}(y)}-\titsm{x}{y}|=|\titsm{\ultra{\omega}(x)}{\ultra{\omega}(y)}-\titsm{h_{n}(x)}{h_{n}(y)}|
  \end{equation*}
Since this is true for every $n$, $\ultra{\omega}$ is distance preserving.
\end{proof}
We finally conclude this section with the:
\begin{prop}
     Suppose $G$ acts geometrically on a Tits compact \catz\ space $X$. Then \titsb{X} is a metric unit sphere, \can{X} is flat, and $G$ is a Bieberbach group.
\end{prop}
\begin{proof}
  According to the proposition \ref{prop:bndrysphere} and Bieberbach' s theorem (see the Theorem $4.2.2$ in the beautiful book \cite{thurston1997three}), it is enough to show that \titsb{X} is a unit sphere. Theorem \ref{thm:gural} gives us a ultra-filter $\omega\in\beta G$ and a metric sphere $\SS\subset\titsb{X}$ such that $\ultra{\omega}$ maps \bndry{X} surjectively onto $\SS$. According to the previous lemma it is also distance preserving, so it is an isometry.
\end{proof}
\section{Dynamics of hyperbolic isometries and Tits compactness}\label{sec:hypcomp}

In section \ref{sec:geodcomp} we will have to distinguish between two possibilities: either the full isometry group fixes a point at infinity or it contains a hyperbolic element. This section will be useful in the later case.
\begin{lemma}\label{lem:split}
Suppose $X$ is a Tits compact \catz\ space that admits a hyperbolic isometry $a$ and choose an axis $c$ of $a$. Consider a ray $r$ emanating from $c$ and different from $c$ itself (we assume such a ray exists). Then $r$ lies in some half flat plane with edge $c$.
\end{lemma}
\begin{proof}
  Let's write $b$ for the restriction of $a$ to the boundary at infinity \bndry{X}, and consider the sequence of maps $\left(b^{-n}\right)_{n\in\N}$. This is a sequence of isometries for the Tits metric which is assumed to give rise to a compact topology. According to the Arzelà-Ascoli theorem, we can fix a series of strictly increasing integers $n_k$ such that $\left(b^{-n_k}\right)$ converges uniformly to a continuous map $b$ of \bndry{X}. Now this map must be distance preserving as a pointwise limit of such maps. But a distance preserving map of a compact metric space is always surjective, and so $b$ is a (surjective) isometry of \titsb{X}. This will prove useful later. Write $d$ for the minimum displacement of $a$ i.e. $d=\inf_{x\in X} d(x,a(x))$.

Consider  an asymptotic direction $x$ different from $c(\pm\infty)$. Then we know (see Part \Roman{deux} proposition 9.8 in \cite{BridsonHaefliger201011}) that the angles $\angle_{c(d\cdot n_k)}\left(c(+\infty),x\right)$ increase to the Tits angle $\angle(c(+\infty),x)$:
\begin{equation}\label{eq:part0}
  \lim_{k\rightarrow+\infty}\angle_{c(d\cdot n_k)}(c(+\infty),x)=\angle(c(+\infty),x).
\end{equation}
Now if we push everything forward in the left hand side of this equality via the isometry $a^{-n_k}$ we get:
\begin{equation}\label{eq:part1}
  \lim_{k\rightarrow+\infty}\angle_{c(0)}(c(+\infty),b^{-n_k}(x))=\angle(c(+\infty),x).
\end{equation}
But the sequence of maps $b^{-n_k}$ also converge pointwise to $b$ in the cone topology, so we have:
\begin{equation}\label{eq:part2}
  \lim_{k\rightarrow+\infty}\angle_{c(0)}(c(+\infty),b^{-n_k}(x))=\angle_{c(0)}(c(+\infty),b(x)).
\end{equation}
and since $b$ is an isometry that fixes $c(+\infty)$ (because all the $b^{-n_k}$ do) we also have:
\begin{equation}\label{eq:part3}
\angle(c(+\infty),x)=\angle(c(+\infty),b(x)).
\end{equation}
Combining, \eqref{eq:part1}, \eqref{eq:part2} and \eqref{eq:part3} we finally get:
\begin{equation*}
\angle_{c(0)}(c(+\infty),b(x))=\angle(c(+\infty),b(x)).
\end{equation*}
But this equality implies that $c(0)$, $c(+\infty)$ and $b(x)$ span an infinite flat triangle, and this for every $x$ (the case where $x$ is either of $c(+\infty)$ and $c(-\infty)$ being trivial). Now by surjectivity of $b$ we just proved that given any asymptotic direction $y\in\bndry{X}$, $c(0)$, $c(+\infty)$ and $y$ span an infinite flat totally geodesic triangle. We call $\Delta_y$ this triangle.

Now remark that the parametrization of $c$ was chosen arbitrarily, so the previous results still holds if we replace $c$ by $c^\prime:t\rightarrow c(t-1)$. In particular, whenever we have an asymptotic direction $y\in\bndry{X}$, the three points $c^\prime(0)=c(-1)$, $c^\prime(+\infty)=c(+\infty)$ and $y$ span convex geodesic flat triangle. We write $\Delta^\prime_y$ for this triangle. But then the euclidean triangle $\Delta^\prime_y$ contains asymptotic rays to $y$ emanating from $c(t)$ for every $t\geq 0$. Since those rays are uniquely determined and lie in $\Delta_y$, we conclude that $\Delta^\prime_y$ extends $\Delta_y$. Proceeding recursively, we then obtain a totally geodesic flat half plane $\Pi_y$ with border $c$ and $y$ on its boundary.

Now if $r$ is the ray of the proposition, it suffices to take $y=r(\infty)$.
\end{proof}
The same method as in the proof of proposition \ref{prop:bndrysphere} then gives:
\begin{prop}\label{prop:hypflat}
 Suppose $X$ is a cocompact Tits compact \catz\ space that admits a hyperbolic isometry. Then $\can{X}$ splits with a euclidean factor.
\end{prop}
\begin{proof}
  We can always assume that $X$ is boundary minimal. We use the notations of the above lemma. Then the set of parallel lines to $c$ (remark it closed and convex) has full boundary in $X$ (this is the lemma \ref{lem:split}) and splits with a \R\ factor. By minimality, so does $X$. 
\end{proof}
\begin{rmk}\label{rmk:piconv}
  We could have used the $\pi$-convergence property, i.e. lemma $18$ of \cite{papasoglu2009boundaries}, to show that $\bndry{X}$ splits as a join $\mathcal{S}^{0}\star Y$. Let us sketch the proof. Since the limit points $n$ and $p$ (we use the notations of \cite{papasoglu2009boundaries}) are at distance not lower than $\pi$ apart here (since they are joined by the axis of the hyperbolic isometry), the $\pi$-convergence property easily shows that for each point $a\in\bndry{X}$, we have $d(n,p)=d(n,a)+d(a,p)$ (and in fact $d(n,p)=\pi$). Now we can use the lemma $3.19$ of \cite{guralnik2011atransversal} to conclude.
\end{rmk}
\section{Dynamics of parabolic groups and Tits compactness}

In this section we study \catz\ spaces that admit a cocompact action by a group of isometries fixing a point at infinity. We were not able to show the existence of hyperbolic elements. However, we construct a series of isometries $\left(g_n\right)_{n\in\N}$ that behaves much like iterates of a single hyperbolic isometry. Then a refinement of the argument used in section \ref{sec:hypcomp} will give us the desired flat factor.
\begin{prop}\label{prop:veryhyp}
  Let $X$ be a \catz\ space and $G$ a group of isometries that acts cocompactly on $X$ and fixes a point $p$ on the boundary \bndry{X}. Then there exists a series $\left(g_n\right)_{n\in\N}$ of isometries in $G$, a geodesic line $R$ and a real $d\geq0$ such that:
  \begin{enumerate}
  \item\label{it:geodline} $R(-\infty)=p$,
  \item\label{it:distatmost} For every integer $n$, $g_n(R)$ is parallel to $R$ and lies at distance at most $d$ to it,
  \item\label{it:limitpt} $\left(g_n^{\pm}(q)\right)_{n\in\N}$ converges to $R(\pm\infty)$ for every $q\in X$.
  \end{enumerate}
\end{prop}
\begin{proof}
  It is easy to construct a geodesic line $R$ satisfying \ref{it:geodline}. Now  fix $x=R(0)$ and $d>0$ such that the translates of the ball $\ball{x}{d}$ by all isometries of $G$ covers $X$.

Take $n\in\N$ and let us construct the isometry $g_n$ above. For each integer $k>0$, fix an isometry $h_{n,k}$ in $G$ that takes $R(k)$ within $d$ of $R(k+n)$. Consider a reparametrization $L_{n,k}$ of $h_{n,k}(R)$ such that $L_{n,k}(0)$ and $x$ belong to the same horosphere centered at $p$ (this is possible because these two lines are asymptotic to $p$). Classically, the function $f_{n,k}:t\mapsto d(R(t),L_{n,k})$ is increasing (it is convex with a finite limit at $-\infty$) and so:
\begin{figure}[!h]
  \centering
  \begin{tikzpicture}

\clip (-3.50,-2.50) rectangle (1.50,3.00);
\draw (0.00,-2.50) -- (0.00,3.00);
\draw (-1.00,-2.50) .. controls (-1.00,0.50) .. (-3.00,3.00);
\draw (-0.50,-16.99) circle (15.00);
\draw (-0.05,0.30) -- (0.05,0.30);
\draw (-0.05,2.00) -- (0.05,2.00);
\draw (0.00,0.30) node[anchor=west] { $R(k)$ };
\draw (0.00,2.00) node[anchor=west] { $R(n+k)$ };
\draw[->] (0.10,-2.10) -- (0.10,-2.50);
\draw (0.10,-2.30) node[anchor=west] { $p$ };
\draw[->] (-1.10,-2.10) -- (-1.10,-2.50);
\draw (-1.10,-2.30) node[anchor=east] { $p$ };
\draw (0.00,-2.00) node[anchor=south west] { ${x=R(0)}$ };
\draw (-1.00,-2.00) node[anchor=south east] { ${L_{n,k}(0)}$ };
\draw (-2.59,2.41) -- (-2.52,2.48);
\draw (-2.56,2.44) node[anchor=south west] { ${h_{n,k}(R(n+k))}$ };
\draw[<->,dashed] (-2.46,2.43) -- (-0.10,2.02);
\draw (-1.28,2.22) node[anchor=north] { $\leq d$ };
\end{tikzpicture}
  \caption{Definition of $h_{n,k}$.}
  \label{fig:hnk}
\end{figure}
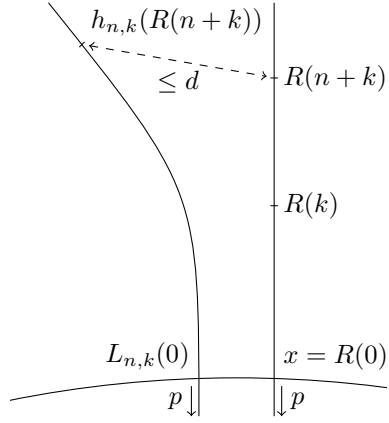
\begin{equation}\label{eq:lineconv}
\forall t\leq n+k,\ f_{n,k}(t)\leq f_{n,k}(n+k)\leq d(R(n+k),h_{n,k}(R(k)))\leq d.
\end{equation}
Moreover $h_{n,k}$ sends $x=R(0)$ to $L_{n,k}(\gamma_{n,k})$ where $\gamma_{n,k}$ is the image of $h_{n,k}$ under the Busemann character based at $p$, i.e. $\gamma_{n,k}=\buse{p}{y}{h_{n,k}(y)}$ for any $y\in X$. But then (Busemann functions are $1$-lipschitz):
\begin{align*}
  \gamma_{n,k}=\buse{p}{R(k)}{h_{n,k}(R(k))}&\leq d(R(k),h_{n,k}(R(k)))\\
  &\leq d(R(k),R(n+k))+d(R(n+k),h_{n,k}(R(k)))\\
  &\leq n+d,
\end{align*}
and so:
\begin{equation}\label{eq:majo}
  d(x,h_{n,k}(x))\leq \underbrace{d(x,L_{n,k}(0))}_{=f_{n,k}(0)}+\underbrace{d(L_{n,k}(0),L_{n,k}(\gamma_{n,k}))}_{=\gamma_{n,k}}\leq d+n+d.
\end{equation}
This bound \eqref{eq:majo} does not depend on $k$. By properness of $X$ we can hence take a converging subsequence of $\left(h_{n,k}\right)_{k\in\N}$. We choose $g_n$ as the limit of this series. Now the condition \ref{it:distatmost} of the proposition is obtained by letting $k$ go to infinity in \eqref{eq:lineconv}. Let us prove that the condition \ref{it:limitpt} also holds. The cocycle relation for Busemann functions gives:
\begin{equation*}
\buse{p}{x}{h_{n,k}(x)}=\buse{p}{x}{R(n+k)}+\buse{p}{R(n+k)}{h_{n,k}(x)},
\end{equation*}
and hence:
\begin{equation*}
\gamma_{n,k}=\buse{p}{x}{h_{n,k}(x)}\geq\buse{p}{x}{R(n+k)}-d(R(n+k),h_{n,k}(x))\geq n-d.
\end{equation*}
Going to the limit in $k$ then gives $\gamma_{n}=\buse{p}{x}{g_{n}(x)}\geq n-d$. But the pythagorean Theorem in the flat strip spanned by $R$ and $g_n(R)$ gives (here $\delta$ is the Hausdorff distance between $R$ and $g_n(R)$):
\begin{equation}\label{eq:cond3foward}
  d(x,g_n(x))=\sqrt{\gamma_{n}^2+\delta^2}\geq\gamma_{n}\geq n-d.
\end{equation}
We also have $d(R,g_{n}^{-1}(R))=d(R,g_n(R))=\delta$ and $\buse{p}{x}{g_{n}^{-1}(x)}=-\gamma_k$ so that we have the same bound with $g_n^{-1}$ instead of $g_n$:
\begin{equation}\label{eq:cond3backward}
  d(x,g_{n}^{-1}(x))=\sqrt{\gamma_{n}^2+\delta^2}\geq\gamma_{n}\geq n-d.
\end{equation}
The conditions \ref{it:limitpt} immediately follows from \eqref{eq:cond3foward} and \eqref{eq:cond3backward} and the fact that $g_n(x)$ and $g_{n}^{-1}(x)$ remain at distance at most $d$ from $R$.
\end{proof}
We will make use of the dynamics of those series. First we will need the following very handful lemma:
\begin{lemma}\label{lem:limtitspar}
  Let $X$ be a \catz\ space with a geodesic line $R$, $y\in\bndry{X}$ and $d>0$ be arbitrary. Suppose we have a sequence of lines $R_n$ that lie at distance at most $d$ from $R$ in the Hausdorff metric (in particular they are parallel to it). Then for any sequence of reals $t_k$  such that $R_k(t_k)$ diverges to $R(+\infty)$ , the angles $\angle_{R_k(t_k)}(R(+\infty),y)$ converge to the Tits angle $\angle(R(+\infty),y)$.
\end{lemma}
\begin{rmk}
  The case where the $R_n$ all coincide with $R$ is classical (see Part \Roman{deux} proposition $9.8$ of \cite{BridsonHaefliger201011}) and we will make use of it to simply the proof.
\end{rmk}
\begin{proof}
  As pointed out in the above remark, we already know that the angles $\angle_{R(t)}(R(+\infty),y)$ converge to $\angle(R(+\infty),y)$ when $t$ goes to $+\infty$. Let $\epsilon>0$ be arbitrary. We now fix $t$ such that  $\angle_{R(t)}(R(+\infty),y)\geq \angle(R(+\infty),y)-\epsilon/2$ and write $x$ for $R(t)$. Let us parametrize the problem (we will omit the subscript $k$ to simplify the notations since we will only consider one $k$ at a time) and define (see figure \ref{fig:layoutangles}):
  \begin{align*}
  \alpha&=\angle_{R_k(t_k)}(R(+\infty),y),&   \theta &=\angle_{R_k(t_k)}(y,x),\\
    z_2 &=\angle_{R_k(t_k)}(x,R(-\infty)),&     z_1  &=\angle_x(R_k(t_k),R(+\infty)),\\
    \varphi&=\angle_x(R_k(t_k),y), &=\beta=\angle_{x}(R(+\infty),y).
  \end{align*}
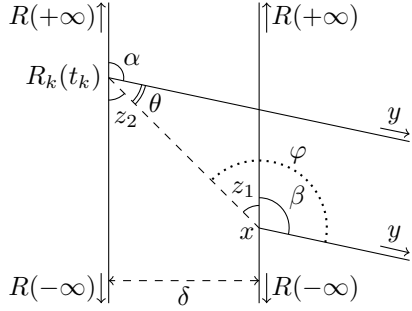
\begin{figure}[h]
  \centering
  \begin{tikzpicture}

\draw (0.00,-2.00) -- (0.00,2.00);
\draw (-2.00,-2.00) -- (-2.00,2.00);
\draw (-2.00,1.00) -- (2.00,0.15);
\draw (0.00,-1.00) -- (2.00,-1.43);
\draw (-2.00,1.20) arc (90.00:-12.00:0.20);
\draw[dashed] (-2.00,1.00) -- (0.00,-1.00);
\draw[<->,dashed] (0.00,-1.70) -- (-2.00,-1.70);
\draw (-1.56,0.91) arc (-12.00:-45.00:0.45);
\draw (-1.51,0.90) arc (-12.00:-45.00:0.50);
\draw (-1.79,0.79) arc (-45.00:-90.00:0.30);
\draw (0.39,-1.08) arc (-12.00:90.00:0.40);
\draw (0.00,-0.70) arc (90.00:135.00:0.30);
\draw[dotted,thick] (0.88,-1.19) arc (-12.00:135.00:0.90);
\draw (-1.77,0.45) node { $z_2$ };
\draw (-1.69,1.25) node { $\alpha$ };
\draw (-1.38,0.67) node { $\theta$ };
\draw (-0.21,-0.49) node { $z_1$ };
\draw (0.51,-0.59) node { $\beta$ };
\draw (-1.00,-1.90) node { $\delta$ };
\draw (-2.60,1.00) node { $R_k(t_k)$ };
\draw (-0.16,-1.07) node { $x$ };
\draw (0.52,-0.03) node { $\varphi$ };
\draw[->] (1.61,0.33) -- (2.00,0.25);
\draw (1.80,0.49) node { $y$ };
\draw[->] (1.61,-1.24) -- (2.00,-1.33);
\draw (1.80,-1.08) node { $y$ };
\draw[->] (0.10,1.60) -- (0.10,2.00);
\draw (0.75,1.80) node { $R(+\infty)$ };
\draw[->] (-2.10,1.60) -- (-2.10,2.00);
\draw (-2.75,1.80) node { $R(+\infty)$ };
\draw[->] (-2.10,-1.60) -- (-2.10,-2.00);
\draw (-2.75,-1.80) node { $R(-\infty)$ };
\draw[->] (0.10,-1.60) -- (0.10,-2.00);
\draw (0.75,-1.80) node { $R(-\infty)$ };
\end{tikzpicture}
  \caption{Layout of the angles.}
  \label{fig:layoutangles}
\end{figure}
Let $\delta\leq d$ be the Hausdorff distance between the two rays $R_k$ and $R$. Then euclidean geometry in the flat strip spanned by those two rays gives:
\begin{align*}
  \sin(z_1)=\sin(z_2)=\delta/d(R_k(t_k),x)\leq d/d(R_k(t_k),x)\xrightarrow[+\infty]{k} 0.
\end{align*}
So if we take $k$ large enough then we have $z_1,z_2\leq\epsilon/4$. We will assume this condition is satisfied in the following. But then:
\begin{equation}\label{eq:majorpi}
  \pi=\angle_{R_k(t_k)}(R(+\infty),R(-\infty))\leq \alpha+\theta+ z_2.
\end{equation}
In the infinite triangle with vertex $R_k(t_k)$, $x$ and $y$ the sum of the angles must not exceed $\pi$ so $\theta+\varphi\leq\pi$. But $\varphi\geq \beta-z_1$, so combining this with \eqref{eq:majorpi} gives:
\begin{equation*}
\theta+\beta-z_1\leq\pi\leq\alpha+\theta+z_2,
\end{equation*}
and  finally we have:
\begin{align*}
  \alpha\geq \beta-(z_1+z_2)\geq \angle(R(+\infty),y)-\epsilon/2-(\epsilon/4+\epsilon/4)\geq \angle(R(+\infty),y)-\epsilon.
\end{align*}
Since we always have $\alpha\leq\angle(R(+\infty),y)$ we just proved the lemma.
\end{proof}
The proof of the following proposition is a refinement of that of the lemma \ref{lem:split}.
\begin{prop}\label{prop:fixflat}
  Let $X$ be a \catz\ space with compact Tits topology. Suppose that a group $G$ fixes a boundary point $p\in\bndry{X}$ and acts cocompactly on $X$. Then $\can{X}$ admits a flat factor.
\end{prop}
\begin{proof}
 We can always assume that $X$ is boundary minimal. Fix a sequence of isometries $g_n$ and a ray $R$ as given by the proposition \ref{prop:veryhyp} above. Let also $x$ lie on $R$. Passing to a subsequence if necessary, we can assume that the isometries ${g_k}_{|\bndry{X}}^{|\bndry{X}}$ converge uniformly to a distance preserving map $b:\titsb{X}\rightarrow\titsb{X}$. Of course, $b$ fixes $R(+\infty)$. Fix $y\in\bndry{X}$ arbitrary. Then according to the lemma \ref{lem:limtitspar} above we have:
 \begin{equation*}
   \lim_{k\mapsto+\infty}\angle_{g_k(x)}(R(+\infty),y)=\angle(R(+\infty),y).
 \end{equation*}
If we push everything foward in the left hand side by $g_{k}^{-1}$ then we obtain:
\begin{equation}\label{eq:eq1}
   \lim_{k\mapsto+\infty}\angle_{x}(R(+\infty),g_{k}^{-1}(y))=\angle(R(+\infty),y).  
\end{equation}
But we obviously have:
\begin{equation}\label{eq:eq2}
   \lim_{k\mapsto+\infty}\angle_{x}(R(+\infty),g_{k}^{-1}(y))=\angle_{x}(R(+\infty),b(y)),
\end{equation}
and since $b$ fixes $R(+\infty)$ we also have:
 \begin{equation}\label{eq:eq3}
   \angle(R(+\infty),y)=\angle(R(+\infty),b(y)).
 \end{equation}
Combining \eqref{eq:eq1}, \eqref{eq:eq2} and \eqref{eq:eq3} we get:
\begin{equation*}
  \forall y\in\titsb{X},\ \angle_{x}(R(+\infty),b(y))=\angle(R(+\infty),b(y)).
\end{equation*}
Now $b$ is surjective since every distance preserving map of a compact set is, so finally we obtain:
\begin{equation}
  \forall z\in\bndry{X},\ \angle_{x}(R(+\infty),z).=\angle(R(+\infty),z).
\end{equation}
The end of the proof is word for word that of the proposition \ref{prop:hypflat}.
\end{proof}
\begin{rmk}
  The author first thought that the $\pi$-convergence property failed here because the $g_n$ constructed in the proposition \ref{prop:veryhyp} need not act properly on $X$. However the author recently realised that the properness hypotheses in the lemma $8$ of \cite{papasoglu2009boundaries} is not used in the proof! This observation yields another proof (see the remark \ref{rmk:piconv}).
\end{rmk}
\section{Geodesically complete Tits compact cocompact spaces are flat}\label{sec:geodcomp}

We shall now use our results to prove the following:
\begin{prop}\label{prop:geodcompflat}
  Every geodesically complete cocompact and Tits compact space is flat. 
\end{prop}
\begin{rmk}
  We suspect that the geodesic completeness is unnecessary (in section \ref{sec:geomact} above and in its improvement section \ref{sec:geomactbetter} bellow this hypothesis is exchanged with another one). The other hypotheses however cannot be dropped. 
\end{rmk}
To prove this proposition we will make use of the following result of \cite{caprace2009isometry} (the result in their is actually a little better):
\begin{thm}\label{thm:capmon}
  A geodesically complete cocompact \catz\ space $X$ not reduced to a singleton admits hyperbolic isometries or admits a cocompact action by a group of isometries fixing a point at infinity.
\end{thm}
The proof of the above theorem is rather evolved since it makes use of the structure of locally compact groups with trivial amenable radical and that of totally disconnected locally compact groups (indeed it even makes use of the fact due to G. Willis and published in $1994$, \cite{willis1994structure}, that a compact subgroup of a locally compact totally disconnected group lies in a compact open subgroup. The existence of compact open subgroups, due to van Dantzig, was proved in $1936$ in \cite{dantzig1936zurtop}). 
\begin{rmk}\label{rmk:capmon}
  The proof of the proposition $6.8$ in \cite{caprace2009isometry} lacks some details and we take the opportunity to warn the reader. The inequality therein $d(gy, y)\leq d(cg,x , y)\angle_{c_{g,x}}(gx, x)$ does not hold for any $y\in[c_{g,x},x]$ like written but for $y$ between $c_{g,x}$ and $x$ close enough to $c_{g,x}$ as soon as the angle $\angle_{c_{g,x}}(gx, x)$ is not zero. If it is zero, then we have $d(gy, y)\leq d(cg,x , y)/2n$ for $y\in[c_{g,x},x]$ close enough to $c_{g,x}$ so the rest of the argument works equally well in this case. We wish to thank P.E. Caprace for his quick answer concerning this point.
\end{rmk}
\begin{proof}[Proof of the proposition \ref{prop:geodcompflat}]
  Once again we decompose $X$ as $\R^k\times Y$ where $Y$ admits no euclidean factor. We want to show that $Y$ is reduced to a point. Suppose not. Then according to the theorem \ref{thm:capmon} above $Y$ either admits a hyperbolic isometry or a cocompact action by a group of isometries fixing a point at infinity. In either case it must admit a flat factor (remark that a geodesically complete cocompact \catz\ space is boundary minimal), despite its definition. So $Y$ is reduced to a point and $X$ is flat.
\end{proof}
\section{Proper semi-simple actions with locally finite stabilisers}\label{sec:geomactbetter}

We use the argument of P.E. Caprace and N. Monod in order to replace the geodesic completness hypothesis in the proposition \ref{prop:geodcompflat} by the existence of a proper semi-simple actions with locally finite stabilisers (definitions below). In this section, we say say that a topological group $G$ acting continuously on a metric space $(X,d)$ acts properly if for every closed ball $B$ in $X$, the set of all $g\in G$ satisfying $g(B)\cap B\neq\empty$ is compact in $G$ (this coincide with the previous definition when $G$ has a discrete topology). With say that $G$ acts with locally finite stabilisers if for every $x\in X$ and $r>0$, the stabiliser $G_{x}$ acts as a finite group of transformation on \ball{x}{r} (i.e. the quotient of $G_x$ by the pointwise fixator of \ball{x}{r} is finite). Then:
\begin{prop}
  Let $G\subset\isom{X}$ acts properly and cocompactly by semi-simple isometries with locally finite stabilisers on a \catz\ space without fixing a point on the boundary. Then it admits a hyperbolic isometry.
\end{prop}
\begin{proof}
  It is essentially that of the corollary $6.10$ in \cite{caprace2009isometry} (we only replace the use of van Dantzig' s theorem with the refinement of G. Willis). The only case that might cause problems is when $G$ is totally disconnected and we shall assume henceforth that it is. We only need to show that the proposition $6.8$ in there still holds. This is not clear at first since it is there made use of the smoothness of the action. We use the notations in the proof of the proposition $6.8$. Obviously the following still holds: if an isometry $g$ of any complete \catz\ space $B$ has order not bigger than $n$, then $\angle_{c_{g,x}}(gx,x)\geq 1/n$ (see also the remark \ref{rmk:capmon} above).

Like in the paper of P.E. Caprace and N. Monod, let us suppose the contrary. By cocompactness we obtain a sequence of isometries $g_n$ of $X$ that converge to an isometry $g$ and for each integer $n$ a point $x_n$ not fixed by $g_n$ such that $c_{g_n,x_n}$ converge to some $c\in X$ and $\angle_{c_{g_n,x_n}}(g_n x_n, x_n)\rightarrow 0$. Since the angle $\angle_{c_{g_n,x_n}}(g_n x_n, x_n)$ depends only on the geodesic segment $[c_{g_n,x_n},x_n]$ we can further assume that $d(c_{g_n,x_n},x_n)\leq 1$. Of course, $g$ fixes $c$, i.e. $g\in G_{c}$ and $G_c$ is compact. By Willis' theorem on the structure of totally disconnected locally compact groups in \cite{willis1994structure}, $G_c$ is contained in a compact open subgroup of $G$ : there exists a point $\hat{c}\in X$ such that $G_{\hat{c}}$ is open in $G$ and $G_c\subset G_{\hat{c}}$. In particular for $n$ large enough $g_n$ fixes $\hat{c}$. Let $d=2+d(\hat{c},c)$. By assumption the group $G_{\hat{c}}$ acts as a finite group of isometries on \ball{\hat{c}}{d} so up to taking a subsequence we can assume all the $g_n$ coincide on this ball (we write $g$ for this restriction) and have finite order $m$ (since $g$ lies in a finite group). Hence $\angle_{c_{g_n,x_n}}(g_n x_n, x_n)=\angle_{c}(g x_n, x_n)\geq 1/m$. This is the desired contradiction.
\end{proof}
\begin{rmk}
  If all the stabilisers are finite (and not merely locally finite) then the proof is easier. Indeed $G$ admits a compact open subgroup $K$ by the structure of totally disconnected locally compact group. But by compactness such a group lies in the stabilizator of a point so must be finite by hypothesis, hence discrete. We just found an open discrete subgroup of $G$: $G$ must itself be discrete. But then Swenson's result (see Theorem $11$ in \cite{swenson1999cut}), whose proof only uses basic \catz\ geometry, applies.
\end{rmk}
The same proof as the one of the proposition \ref{prop:geodcompflat} (after restricting ourselves to the canonical boundary minimal subspace) then gives:
\begin{prop}\label{prop:semisimlocalfinflat}
  Let $X$ be a Tits compact \catz\ space and suppose a topological group $G$ acts properly and cocompactly by semi-simple isometries with locally finite stabilisers on $X$. Then \can{X} is flat, or equivalently \titsb{X} is a unit metric sphere.
\end{prop}

\section{On the existence of hyperbolic isometries}

We include two results than show that the difficulty in finding hyperbolic isometries in a cocompact \catz\ space $X$ is related to this existence of flats in $X$. In particular a cocompact \catz\ hyperbolic space always admits hyperbolic isometries. We begin by the easiest one:
\begin{prop}
  Let $G$ be a proper group of isometries of a \catz\ space $X$ acting cocompactly and fixing a point $p$ on the boundary. Then either $G$ possesses hyperbolic elements or there exists a flat in $X$ with $p$ as boundary point.
\end{prop}
 Under the hypothesis of the proposition we know that there exists a geodesic line $R$ asymptotic to $p$. The parallel set to $R$ splits as $\R\times Y$ for some \catz\ space $Y$. Then:
\begin{lemma}\label{lem:coc}
  $Y$ is cocompact.
\end{lemma}
\begin{proof}[Proof of the lemma]
  Let $d>0$ be such that the translates by $G$ of any ball of radius $d$ covers $X$. Let $a,b\in Y$ be arbitrary, and note $\varphi_a$ and $\varphi_b$ geodesic lines asymptotic to $p$ corresponding to  $a$ and $b$ respectively and whose origin are on the same horosphere based at $p$. For all $k\in\N$ let $h_k$ be an isometry of $X$ that sends $\gamma_a(k)$ within $d$ to $\gamma_b(k)$. The same type of comparisons we used to construct $g_n$ in the proof of the proposition \ref{prop:veryhyp} show that $h_k(\gamma_a(0))$ remains uniformly bounded, and hence the series of isometries $h_k$ has a convergence subsequence. Let $h$ be a limit of such a subsequence. Then $h$ is an isometry that sends $R$ to a parallel line that lies within $d$ of $\gamma_a$. Hence $h$ stabilizes \para{R}, respects its product decomposition $\R\times Y$, and sends $a$ in $Y$ to within $d$ of $b$. Since $b$ is arbitrary we just proved the cocompactness of $Y$.
\end{proof}
\begin{lemma}\label{lem:isomexist}
  There exists an isometry of $X$ preserving the parallel set of $R$ and with non zero Busemann character at $p$.
\end{lemma}
\begin{proof}
  It is the same construction as in the previous lemma (with $a=b$ arbitrary) except we choose $h_k$ that sends $\gamma_a(k)$ to within $d$ of $\gamma_a(k+2d)$ to force the Busemann character to be positive.
\end{proof}
\begin{proof}[Proof of the proposition]
  Let $\R\times Y$ be the parallel set to $c$.  If $Y$ is bounded then it admits a circumcenter $c$. Let $h$ be an isometry given by the lemma \ref{lem:isomexist}. Then $h$ stabilises the line corresponding to $\R\times \{c\}$ in the decomposition of \para{R}. Since the Busemann character of $h$ is non trivial, $h$ acts as a non degenerates translation on this line (i.e. it the translation length is not zero). But then $h$ is hyperbolic.

 Suppose now that \para{R} is not bounded. Then it admits a geodesic line $L$ since it is cocompact by lemma \ref{lem:coc} and $L\times R$ is the desired flat.
\end{proof}
To simplify the proof of the next proposition, we begin with a definition and a lemma:
\begin{defi}
  Let $\theta$ and $\psi$ be two asymptotic rays in a \catz\ space. Then we define $\distinf{\theta}{\psi}$ by:
  \begin{equation*}
    \distinf{\theta}{\psi}=\lim_{t\rightarrow+\infty}d(\theta(t),\hat{\psi}(t)),
  \end{equation*}
where $\hat{\psi}$ is a reparametrization of $\psi$ such that $\psi(t)$ and $\theta(t)$ always lie on the same horosphere based at $\psi(+\infty)=\theta(+\infty)$ (hence it does not depend on the unit-speed parametrisation of $\theta$ and $\psi$, but might not depend solely on their asymptotic class).
\end{defi}
\begin{lemma}\label{lem:distinfdiv}
  Let $X$ be a cocompact \catz\ space. Suppose there exists for every $n\in\N$ two asymptotic rays $\theta_n$ and $\psi_n$ with $\distinf{\theta_n}{\psi_n}\geq n$. Then $X$ contains a flat.
\end{lemma}
\begin{proof}
  By cocompacity, it is enough to find flat half strips of arbitrary width. Take $n\in\N$ and the corresponding rays $\theta_n$ and $\psi_n$. For each $k\in\N$ let $h_{n,k}$ be an isometry that sends $\theta_n(k)$ to within a fixed distance of a fixed point. Up to passing to a subsequence we can assume that the rays $h_{n,k}(\theta_n)$ and $h_{n,k}(\psi_n)$ converge to asymptotic rays $\theta_{n}^\prime$ and $\psi_{n}^\prime$. In fact, the distance between $\theta_{n}^\prime$ and $\psi_{n}^\prime$ is constant equal to \distinf{\theta_n}{\psi_n}, so that they bound a half flat strip of width at least $n$ and we are done.
\end{proof}
Here comes the long expected proposition:
\begin{prop}
  Let $X$ be a cocompact \catz\ space. Then either it admits a hyperbolic isometry or it contains a flat.
\end{prop}
\begin{cor}
  Let $X$ be a Gromov hyperbolic cocompact \catz\ space. Then it admits a hyperbolic isometry.
\end{cor}
\begin{rmk}
  The corollary and the proposition are in fact equivalent statement by the theorem A of \cite{bridson1995existence}.
\end{rmk}
\begin{proof}
\begin{figure}[h]
  \centering
  \begin{tikzpicture}

\draw (0.00,2.00) -- (0.00,-2.00);
\draw (0.00,2.00) -- (-2.35,1.24);
\draw (0.00,2.00) -- (2.35,1.24);
\draw (0.00,-2.00) -- (-2.35,1.24);
\draw (0.00,-2.00) -- (2.35,1.24);
\draw (-0.29,1.91) arc (-162.00:-90.00:0.30);
\draw (2.17,0.99) arc (-126.00:-198.00:0.30);
\draw (0.29,1.91) arc (-18.00:-90.00:0.30);
\draw (0.33,1.89) arc (-18.00:-90.00:0.35);
\draw (-2.07,1.33) arc (18.00:-54.00:0.30);
\draw (-2.02,1.34) arc (18.00:-54.00:0.35);
\draw[dotted,thick] (-0.86,1.72) arc (-162.00:-18.00:0.90);
\draw (-0.29,-1.60) arc (126.00:90.00:0.50);
\draw (-0.15,-1.53) circle (0.06);
\draw (0.00,-1.50) arc (90.00:54.00:0.50);
\draw (0.15,-1.53) circle (0.06);
\draw (2.35,1.24) node[anchor=south west] { $r$ };
\draw (0.00,2.00) node[anchor=south] { $h_n(r)$ };
\draw (-2.35,1.24) node[anchor=south east] { $h^{2}_n(r)$ };
\draw (0.00,-2.00) node[anchor=north] { $x_n$ };
\draw (-1.78,1.05) node { $\alpha_n$ };
\draw (-0.35,1.51) node { $\beta_n$ };
\draw (0.71,1.03) node { $\varphi_n$ };
\draw (-0.25,-1.24) node { $\gamma_n$ };
\end{tikzpicture}
  \caption{When $h_n$ is elliptic.}
  \label{fig:twotriangles}
\end{figure}
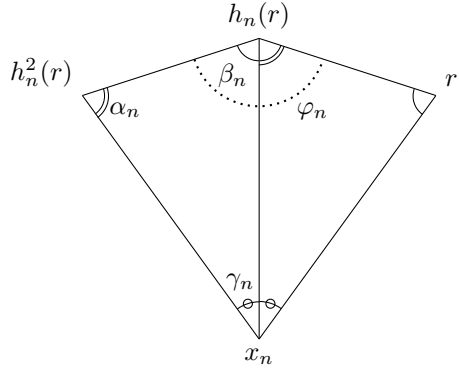
The beginning of the proof is the same as that of the theorem $11$ in \cite{swenson1999cut}, i.e. we get a point $r\in X$ and a family of isometries $h_n\in G$ such that the angle $\varphi=\angle_{h_n(r)}(h^{2}_n(r),r)$ tends to $\pi$ at infinity. We assume moreover that $d(r,h_n(r))$ diverges (we only have to choose the $r_j$ and $r_i$ in the proof of the theorem $11$ of \cite{swenson1999cut} such that $d(r_i,r_j)\rightarrow+\infty$). Suppose that $X$ admits no hyperbolic isometry. We shall distinguish between two cases. 

\textbf{An infinity of the $h_n$ are elliptic: }We can assume all the $h_n$ are elliptic and we fix a fixed point $x_n$ of $h_n$ for each $n$. Then the isosceles triangles $x_nh_n(r)h_{n}^2(r)$ and $x_nrh_n(r)$ are isometric and we label the angles according to the figure \ref{fig:twotriangles}. Now isosceles triangles in the euclidean space have acute base angles. By comparison, this must be the case in all \catz\ spaces, so that we have here $\alpha_n,\ \beta_n\leq\pi/2$. But then:
\begin{align*}
  \underbrace{\angle_{h_n(r)}(h^{2}_n(r),r)}_{\rightarrow\pi}\leq\underbrace{\alpha_n}_{\leq\pi/2}+\underbrace{\beta_n}_{\leq\pi/2},
\end{align*}
and so both $\alpha_n$ and $\beta_n$ must converge to $\pi/2$.

Let us pick $a>0$. Since $d(r,h_n(r))$ diverges, for sufficiently large $n$ we can take $s_n$ (resp $v_n$) on the geodesic joining $x_n$ and $h_n(r)$ (resp. on the geodesic joining $x_n$ and $r$) such that $d(s_n,v_n)=a$ and $d(s_n,x_n)=d(v_n,x_n)$. Let $\alpha_{n}^\prime=\angle_{s_{n}}(x_n,v_n)$ and $\beta_{n}^\prime=\angle_{v_n}(x_n,s_n)$. The item $(3)$ in the proposition $9.8$, Part \Roman{deux},  of \cite{BridsonHaefliger201011} shows that the sum  $\alpha_{n}^\prime+\beta_{n}^\prime$ is not lower that $\alpha_{n}+\beta_{n}$ and we  know that the latter tends to $\pi$. Since we always have $\alpha_{n}^\prime,\ \beta_{n}^\prime\leq\pi/2$ (the triangle $s_nx_nv_n$ is again isosceles), both $\alpha_{n}^\prime$ and $\beta_{n}^\prime$ must converge to $\pi/2$. But then by direct comparison with euclidean geometry, the isosceles triangle $s_nv_nx_n$  becomes arbitrary big. By cocompacity we can take a subsequence $n_k$ such that $s_{n}v_{n}x_{n}$ converges to some $svx$ where $s,v\in X$ and $x\in\bndry{X}$. By upper semi-continuity of the angle (see proposition $9.2$ Part \Roman{deux} in \cite{BridsonHaefliger201011}) we have $\angle_s(v,x)\geq\pi/2$ and $\angle_v(s,x)\geq\pi/2$. But the sum of those two angles cannot be bigger that $\pi$ so we have equality and $svx$ must in fact be a half flat strips of width $a$ as shown for example in \cite{BridsonHaefliger201011}, proposition $9.3$ Part \Roman{deux}. We just constructed half strips of arbitrary big width in $X$. It is then easy to construct a whole flat-plane by cocompacity.
\begin{figure}[h]
  \centering
  \begin{tikzpicture}

\clip (-3.05,-2.70) rectangle (3.05,2.70);
\draw (0.00,2.00) -- (0.00,-2.00);
\draw (0.00,2.00) -- (-2.35,1.24);
\draw (0.00,2.00) -- (2.35,1.24);
\draw (-2.35,1.24) .. controls (-0.50,0.00) .. (-0.50,-2.00);
\draw (2.35,1.24) .. controls (0.50,0.00) .. (0.50,-2.00);
\draw (-0.29,1.91) arc (-162.00:-90.00:0.30);
\draw[dotted,thick] (-0.86,1.72) arc (-162.00:-18.00:0.90);
\draw (0.29,1.91) arc (-18.00:-90.00:0.30);
\draw (0.33,1.89) arc (-18.00:-90.00:0.35);
\draw (0.33,1.89) arc (-18.00:-90.00:0.35);
\draw (0.00,-2.00) circle (4.00);
\draw (-2.07,1.33) arc (18.00:-33.73:0.30);
\draw (-2.02,1.34) arc (18.00:-33.73:0.35);
\draw (2.35,1.24) node[anchor=south west] { $r$ };
\draw (0.00,2.00) node[anchor=south] { $h_n(r)$ };
\draw (-2.65,1.29) node[anchor=south] { $h^{2}_n(r)$ };
\draw (-1.71,1.11) node { $\alpha_n$ };
\draw (-0.35,1.51) node { $\beta_n$ };
\draw (0.71,1.03) node { $\varphi_n$ };
\draw[->] (-0.60,-1.60) -- (-0.60,-2.00);
\draw (-0.60,-1.80) node[anchor=east] { $p_n$ };
\draw[->] (0.60,-1.60) -- (0.60,-2.00);
\draw (0.60,-1.80) node[anchor=west] { $p_n$ };
\end{tikzpicture}
  \caption{When $h_n$ is parabolic.}
  \label{fig:twoinftriangles}
\end{figure}
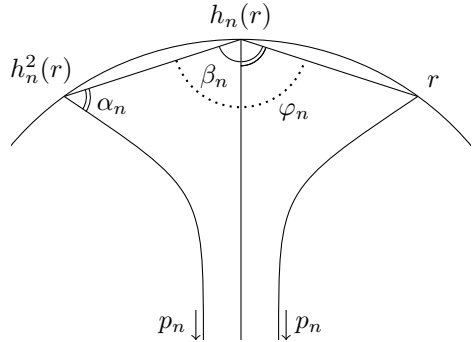

\textbf{An infinity of the $h_n$ are parabolic: }We can assume this time that all the $h_n$ are parabolic. It is shown in \cite{ballmann1985manifolds} that $h_n$ stabilises all the horospheres of some point $p_n$ on the boundary (this need not be the case of all its fixed points). We use the notations of the figure \ref{fig:twoinftriangles} and call $\theta_n$ (respectively $\psi_n$) the ray issuing from $h_n^{2}(r)$ (resp. $r$) and asymptotic to $p_n$. We can assume that the distances at infinity $\distinf{\theta_n}{\psi_n}$ are bounded since otherwise we can apply the lemma \ref{lem:distinfdiv} above. Let $d>0$ be such a bound. Let $a>d$ be chosen arbitrarily. From now on the construction is similar to the case where the $h_n$ are elliptic. Let $s_n$ (respectively $v_n$) be on the ray from $h(r)$ (resp. $r$)  to $p_n$ such that $d(s_n,v_n)=a$ and $s_n$ and $u_n$ both lie on the horosphere based at $p_n$ (so we have an ``infinite isosceles'' triangle). Let $\alpha_{n}^\prime=\angle_{s_{n}}(p_n,v_n)$ and $\beta_{n}^\prime=\angle_{v_n}(p_n,s_n)$. It is proved in the proof of the lemma $4.3$ of \cite{ballmann1985manifolds} (or in the proposition $9.8$ Part \Roman{deux} \cite{BridsonHaefliger201011}) that the sum  $\alpha_{n}^\prime+\beta_{n}^\prime$ is not lower that $\alpha_{n}+\beta_{n}$ and we  know that the latter tends to $\pi$. We still have $\alpha_{n}^\prime,\ \beta_{n}^\prime\leq\pi/2$ (to see this one can either take finite isosceles triangles that converge to $s_nv_np_n$ or use the fact that the horosphere based at $p_n$ through $s_n$ and $v_n$ is convex) and hence  $\alpha_{n}^\prime$ and $\beta_{n}^\prime$ both converge to $\pi/2$. Now by cocompacity we can assume that the triangle $s_nu_np_n$ converges to $sup$ where $s,u\in X$ and $p\in\bndry{X}$. By upper semi-continuity of the angle, $\angle_s(p,u)$ and $\angle_u(p,s)$ cannot be lower than $\pi/2$. Since there sum is lower than $\pi$ they must both be equal to $\pi/2$ and the triangle is in fact a half flat strip of width $a$ (again proposition $9.3$ Part \Roman{deux} of \cite{BridsonHaefliger201011}). We constructed flat strips of arbitrary big width: we conclude like in the previous case. 
\end{proof}
\bibliographystyle{plain}
\bibliography{biblio}
\end{document}